\documentclass[12pt]{article}
\usepackage{amsmath}
\usepackage{amsthm}
\usepackage{amsfonts}
\usepackage{exscale}
\usepackage{enumerate}
\usepackage{hyperref}
\newtheorem{theorem}[equation]{Theorem}

\newtheorem{claim}[equation]{Claim}

\newtheorem{lemma}[equation]{Lemma}

\newtheorem*{hyp*}{\sc Hypothesis $(**)$}

\newcommand{\Q}{{\mathbf{Q}}}

\DeclareMathOperator{\V}{V}

\DeclareMathOperator{\Gal}{Gal}

\DeclareMathOperator{\tensor}\otimes{}

\DeclareMathOperator{\C}{C}
\DeclareMathOperator{\N}{Norm}
\DeclareMathOperator{\T}{Trace}
\DeclareMathOperator{\X}{NormOneGroup}

\newcommand{\QQ}{{\mathbf{Q}}}

\newtheorem{cor}[equation]{Corollary}

\newtheorem{prop}[equation]{Proposition}

\theoremstyle{definition}

\def\bv{\setbox0\vbox\bgroup\hsize4.18in}
\def\ev{\egroup\hskip1.in\box0}

\title{Subgroups of Division Rings}
\author{Mark Lewis \and Murray Schacher}

\begin{document}
\maketitle
\begin{abstract} \label{abs}
 We investigate the finite subgroups that occur
in the Hamiltonian quaternion algebra over the real subfield of cyclotomic
fields.  When possible, we investigate their distribution among the
maximal orders.

\end{abstract}

%\caps{\tt{} }} \\


\par

\section{Introduction} \label{sec:intro}

%We will work in a symbol algebra over various fields of definition.
Let $F$ be a field, and fix $a$ and $b$ to be non-$0$ elements of $F$.
The {\it symbol algebra} $A = (a,b)$ is the $4$-dimensional algebra over
$F$ generated by elements $i$ and $j$ that satisfy the relations:

\begin{equation}\label{stdform}
 i^2 = a, \ \ j^2 = b, \ \ ij = -ji.
\end{equation}

One usually sets $k = ij$, which leads to the additional circular relations
\begin{equation}
 ij = k = -ji, \ \ \ jk = i = -kj, \ \ \ ki = j = -ki.
\end{equation}
%\noindent
The set $\{1,i,j,k\}$ forms a basis for $A$ over $F$.  It is not difficult to see that
$A$ is a central simple algebra over $F$, and using Wedderburn's theorem, we see that $A$ is either a $4$-dimensional division ring or the ring $M_2(F)$ of $2\times 2$ matrices over $F$.  It is
known that $A$ is split (i.e. is the ring of $2 \times 2$ matrices over $F$) if and only if
$b$ is a norm from the field $F(\sqrt a)$.  Note that $b$ is a {\it norm} over $F(\sqrt a)$ if and only if there exist elements $x,y \in F$ so that $b = x^2 - y^2 a$.  This condition is symmetric
in $a$ and $b$.

More generally, we have the following isomorphism of algebras:
\begin{equation}\label{norm}
(a,b) \cong (a,ub)
\end{equation}
\noindent where $u = x^2 - y^2a$ is a norm from $F(\sqrt a)$; see \cite{pierce} or \cite{serre}.  We will usually be concerned with the case $a = b = -1$, where $F$ is the real subfield of a cyclotomic number field.   In this case, it is not difficult to see that $-1$ will not be a norm in $F (\sqrt {-1})$, so this ensures that $A = (-1,-1)$ is a division ring (see \cite{serre} or \cite{artate}).   We note that(\ref{norm}) says there are alternative descriptions of $A$.  When exact calculations are required, the form of the isomorphism in (\ref{norm}) will be an issue.

When $F$ is the rational field $\QQ$, the algebra $A = (-1,-1)$ is often called the
{\it ordinary Hamiltonion quaternion algebra}.  Motivated by this, we will say that $A$ is the {\it quaternion algebra} over $K$ when $K$ is any field.  When $K$ is clear, we will just say that $A$ is a quaternion algebra.
A symbol algebra $A = (a,b)$ has an {\sl involution}, $u \rightarrow \bar u$, where
if $$u = u[0]\cdot 1 + u[1]\cdot i + u[2] \cdot j + u[3] \cdot k, $$ then
\begin{equation}
\bar u = u[0]\cdot 1 - u[1]\cdot i - u[2] \cdot j - u[3] \cdot k.
\end{equation}

The usual rules apply:
\begin{equation}\label{defs}
\bar{\bar u} = u, \ \ \T(u) = u + \bar u \in F, \ \
\N(u) = u \bar u \in F
\end{equation}
\noindent and $\T$ is linear over $F$.  Every element of $A$ satisfies
its characteristic polynomial
\begin{equation}\label{charpoly}
x^2 - \T(u)\cdot  x + N(u)\cdot 1.
\end{equation}

When $A = (-1,-1)$, the function $\N$ is a sum of four squares, and so, it represents a positive definite quadratic form when $F$ is real.  Symbol algebras are a special case of {\sl cyclic algebras}, which we
briefly review next.

\section{Cyclic Algebras} \label{sec:cyc}

Let $K$ be a field, and let $L/K$ be a cyclic field extension of dimension $n$.
Writing $\sigma$ for a generator for $\Gal(L/K)$, we see that $\sigma$ has order $n$.
Let $a \in K^*$ be a non-$0$ element of $K$.  The cyclic algebra
$A = (L/K, \sigma, a)$ is defined to be
$$
A = L\cdot 1 + L \cdot x + L \cdot x^2 + \dots + L \cdot x^{n-1}
$$
as a left vector space over $L$ with the relations:

\begin{itemize}
\item $x^n = a$,
\item $x\alpha = \sigma(\alpha)x$ for all $\alpha \in L$.
\end{itemize}

Thus, $(L/K,\sigma,a)$ is a central simple algebra over $K$ of dimension
$n^2$.  It is the $n \times n$ matrix ring $M_n(K)$ if and only
if $a$ is a norm from $L$ to $K$.  More generally, $(L/K,\sigma,a)
\cong (L/K,\sigma,b)$ if and only if $a\cdot b^{-1}$ is a norm
from $L$ to $K$.  The connection with the previous section is the following; see \cite{pierce}
or \cite{serre} for proofs:

\begin{claim}\label{cyclic} \
Let $K,L$ be fields of characteristic $0$, and suppose $L = K(\sqrt a)$, for some element $a \in K$.
Let $\sigma : \sqrt a \rightarrow -\sqrt a$ be the generating automorphism
of $L/K$.  Then $(L/K,\sigma,b) \cong (a,b)$ for some nonzero element $b \in K$. %We need to define $b$.
\end{claim}

We note that the algebra $(L/K,\sigma,b)$ in Claim \ref{cyclic} will be split if and only if $b$ is a norm in $K$. All our fields will be number fields, and so Proposition \ref{cyclic} will
apply.  A quadratic splitting field $L$ of the quaternion algebra $A = (a,b)$ over
$K$ is a quadratic field extension $L = K(\sqrt u)$ for some element $u \in K$ so that $A \tensor_K L$ is $2 \times 2$ matrices over $L$. Claim \ref{cyclic} has a converse, which we will use often; again see \cite{pierce} or \cite{serre} for proofs:

\begin{claim}\label{split} \
Let $A$ be a quaternion algebra over $K$, and let $L$ a quadratic splitting field of $A$.
Then $L$ is isomorphic to a subfield of $A$, and $A \cong (L/K,\sigma,c)$ for
some element $c \in K$.
\end{claim}

When necessary, we will have to make the isomorphism in Claim \ref{split} explicit.

\section{Finite Subgroups} \label{sec:finite}

From now on, we assume $A = (-1,-1)$ over a field $K$ which is the
real subfield of a cyclotomic field.  By \cite{artate}, $A$ is a
$4$-dimensional division ring central over $K$.  We are interested in the
finite groups $G$ which are subgroups of the multiplicative group of $A$.

Many of our conclusions are dependent on the computer program Magma,
and its QuaternionAlgebra package.  Thus, for convenience, we often
illustrate calculations in Magma format.  However, we do not report version
numbers, and so the reader may have to make adjustments in verifying our
calculations.  Our use of Magma format is a notational convenience.

\par
The subgroups of $A$ are a special case of the results from the monumental paper
of {\hbox {Amitsur \cite{amitsur}}, in which all finite subgroups of division rings in characteristic
zero are characterized.  Suppose $G$ is a finite group contained in $D^*$, the multiplicative group
of a division ring $D$ of characteristic $0$.  We emphasize that $D$ need not be
finite dimensional over $\Q$.  Let $\V (G)$ be the set of $\Q$ linear combinations
of the elements $g \in G$.  Observe that $\V (G)$ is a subring of $D$.  Note that $\V (G)$ is not isomorphic to the group algebra $\Q [G]$ since $\Q [G]$ has zero divisors and $\V [G]$ does not.

Now $G$ acts without fixed points (by left translation) on the division ring $D$, and so $G$ is a Frobenius complement -- a condition imposing strong restrictions.  Recall that a group $H$ is a Frobenius complement if $H$ acts on a group $K$ and satisfies the condition that $C_K (h) = 1$ for all $h \in H \setminus \{ 1 \}$.  Many of the results regarding Frobenius complements can be found in Section V.8 of \cite{hup}.  One of the main theorems of \cite{amitsur} is

\begin{prop}[Amitsur]\label{am}
Let $G$ be a finite subgroup of the the division ring $D$.
\begin{enumerate}
\item $\V (G)$ is independent of $D$, and is a finite dimensional division algebra
over $\Q$.  Its center is an abelian extension of $\Q$.

\item Every Sylow subgroup of $G$ is cyclic or generalized quaternion.
\end{enumerate}
\end{prop}

Item (1) of Proposition \ref{am} is a statement not about $G$ but its representation, i.e.
there is a unique faithful representation into a division ring (and not matrices
over a division ring).  In particular, $\V (G)$  behaves as a minimal model; any division ring which
contains $G$ also contains $\V (G)$, although not necessarily with the same center.
In fact, it can be proved (however, we will not include the proof here) that every
group which is contained in a division ring is contained in one that is central and infinite
dimensional over $\Q$.

Item (2) of Proposition \ref{am} is an immediate consequence of $G$ being a Frobenius complement.  However, the condition of being a Frobenius complement is a much stronger condition than the one listed in (2).  For example, $S_3$ satisfies (2) but is not a Frobenius complement, so $S_3$ will not be a subgroup of a division ring.
%Not every group satisfying item (2) of Proposition \ref{am} is a finite subgroup of a division ring; for example, the group $S_3$ satisfies (2) but does not so appear.  In general, this is true of dihedral groups, which likewise satisfy (2).

Not all groups that are contained in division algebras are contained
in our quaternion algebras.  Suppose $A = (-1,-1)$ over a real number
field $K$.  By \cite{amitsur}, the following are candidates for finite
subgroups $G$ of $A$:

\begin{enumerate}
\item Cyclic Groups
%\item Generalized quaternion groups of order $2^n$, $n > 2$.
\item Generalized quaternion groups $G_{2n}$ of order $2n$, $n$ even.
\item The {\sl binary octahedral group} $B$ of order $48$, also called the Clifford group.  (This is the group of order $48$ that is isoclinic to but not isomorphic to ${\rm GL} (2,3)$.)
\item The {\sl binary icosahedral group} $S$ of order $120$. (I.e., $S \cong {\rm SL}(2,5)$.)
\end{enumerate}

The groups $G_{2n}$ are generated by elements $x$ and $y$ subject to the relations:

\begin{equation}
x^n = 1,\ \ y^4 = 1, \ \ y^2 = x^m,  \ \ m = n/2, \ \ yxy^{-1} = x^{-1}.
\end{equation}

The element $y^2$ generates the center of $G_{2n}$ and has order $2$, and modulo $Z (G_{2n}) = \langle y^2\rangle$, the quotient $G_{2n}/Z(G_{2n})$ is the dihedral group
of order $n$.  This dihedral group is not a subgroup of a division ring.

The group $B$ occurs in $(-1,-1)$ over the field $\Q (\sqrt 2)$.  This can be verified
by Magma, and arises in the following way:  The elements $\{1,i,j,k\}$ generate the quaternion group $\Q_8$ of order $8$.
The element $(1 - i -j - k)/2$ is of order $3$ and normalizes $\Q_8$; together
they generate a group $S$ of order $24$ that is isomorphic to ${\rm SL}(2,3)$.  This much
occurs in $(-1,-1)$ over $\Q$.  Finally, the element $(i - j)/\sqrt{2} \in (-1,-1)_{\Q (\sqrt 2)}$ is of order $4$ and normalizes $S$, and together $S$ and $(i-j)/\sqrt {2}$ generate $B$.

The group $S$ of order $120$ occurs in $(-1,-1)$ over the field $\Q (\sqrt{5})$.
This can be verified by Magma, and by the characterization of the following
theorem.  We intend this theorem as an illustration of (1) of Proposition \ref{am}.

\begin{theorem}\label{sl}
The following are true:
\begin{enumerate}
\item Any division ring containing ${\rm SL}(2,5)$ contains $(-1,-1)$ over $\Q (\sqrt 5)$.
\item Any division ring containing $B$, the Clifford group, contains $(-1,-1)$ over $\Q (\sqrt 2)$.
\end{enumerate}
\end{theorem}

We prove only (1); the proof of (2) proceeds similarly.  Because it is useful in other contexts, we include the following lemma.  Let $C_n$ be the cyclotomic field of $n$-th roots of unity, and let
$\epsilon$ be a primitive $n$-th root of unity in $C_n$.  Let $\sigma$ be the
automorphism of order $2$ that sends $\epsilon$ to $\bar{\epsilon} = \epsilon^{-1}$.
Let $A(n,-1)$  be the cyclic crossed product determined by $C_n$ and $\sigma$,
with $\sigma^2 = -1$.  The center of $A(n,-1)$ is the fixed field $K$ of
$\sigma$, which is the real subfield of $C_n$, and $A(n,-1) = (C_n/K,\sigma,-1)$.
When $n = 10$, we note that the real subfield of $C_{10}$ is the real subfield
of $C_5$, and this is the field $K = \Q (\sqrt 5)$.  Let $(-1,-1)_K$ be the ordinary
quaternion algebra tensored up to this field $K$.

\begin{lemma} \label{change}
If $K = \Q (\sqrt 5)$, then $(-1,-1)_K \cong  A(10,-1)$.
\end{lemma}

Although Magma can directly verify this statement, we give a short proof using the theory of Hasse invariants; see \cite{artate} for details.

\begin{proof}
The ordinary quaternion algebra $(-1,-1)_{\Q}$ has non-$0$ invariant
$ = 1/2$ only at the primes $2$ and infinity.  When tensored up to $K = \Q(\sqrt 5)$, these
invariants must be multiplied by the local degree.  As $K/\Q$ has degree $2$ over
$2$ (it is unramified), the invariant over the prime $2$ vanishes.  Thus,
$(-1,-1)_K$ is ramified only and exactly at the infinite primes of $K$.  It is
easy to see that $A(10,-1)$ has the same property, as so they are isomorphic.
\end{proof}

We note that this argument would apply as well for $K = \Q (\sqrt 2)$, since its
degree at 2 is still even (in this case it is ramified).  We now prove (1) of Theorem \ref{sl}.

\begin{proof}[Proof of Theorem \ref{sl}]
Suppose $G = {\rm SL}(2,5)$ is contained in a division ring $D$.  By Lemma \ref{change}, it is enough to show that $D$ contains $A(10,-1)$.  We know that $G$ is generated by the elements $P$,$Q$ who satisfy the following relations:

\begin{equation}
P^2 = Q^3 = (PQ)^5, \ \ P^4 = 1.
\end{equation}

\noindent In $D^*$, the equation $P^4 = 1$ implies $P^2 = -1$.

Set $R = -Q = P^2 Q = Q^4$, {\hbox{$S = -PQ = PQ^4$}}.  Then, in $D^*$,

\begin{equation}
R^3 = 1, S^5 = 1, (SR^{-1})^2 = -1, R^2 + R + 1 = 0, S^4 + S^3 + S^2 + S + 1 = 0.
\end{equation}

\noindent One verifies that $SR^{-1} = - RS^{-1}$, and since $R^2 + R + 1 = 0$,
$R^{-1} = -R - 1$ and
$-S(R+1) = -RS^{-1}$ which yields $SR = RS^{-1} - S$. As $\V (G)$ is already generated
over $\Q$ by $-1,R,Q,S$, we see that $\V (G)$ is spanned over $\Q$ by the eight elements
$\{Q,SQ,S^2Q,S^3Q,RQ,RSQ,RS^2,RS^3Q\}$.  We obtain the inequality $[\V (G) : \Q] \le 8$.  However,
${\rm SL}(2,5)$ contains the generalized quaternion group of order $20$ (it is the normalizer
of an element of order $5$), and this group is contained in $A(10,-1)$ as the group
generated by $\epsilon$ and $\sigma$.  Hence $A(10,-1) \subset \V(G)$.  As $A(10,-1)$ is
$8$-dimensional over $\Q$, we have equality, and the proof is complete.
\end{proof}

The proof of Theorem \ref{sl} actually shows for $G_{20}$, the generalized
quaternion group of order $20$, that $\V (G_{20}) = \V (SL(2,5))$.  Thus, $\V (G)$
is not uniquely determined by $G$, although it is independent of the division
ring which contains $G$.

Note that ${\rm SL}(2,5)$ is not isomorphic to $S_5$, and $S_5$ is not a subgroup
of a division ring; its $2$-Sylow subgroup does not satisfy (2) of Proposition \ref{am}.
The group ${\rm SL}(2,5)$ modulo its central subgroup of order $2$ is $A_5$, which is also
not a subgroup of a division ring for the same reason that its $2$-Sylow subgroup is not cyclic or generalized quaternion.

\section{Orders and Maximal Orders} \label{sec:orders}

Let $B$ be a quaternion algebra over a number field $L$, and suppose that
$R$ is the ring of integers of $L$.  An {\sl order} $O$ of $B$ is a subring
of $B$, containing the same unit element $1$ of $B$, such that
$O$ is a full (i.e. 4-dimensional) $R$ sublattice of $B$.  A {\sl maximal order} is
an order which is not properly contained in any other order.  There is a discriminant
test for maximal orders, but we will not deal with that here.

If $O$ is an order and $x \in O$, then $R[x]$ (the $R$-linear combinations of the
powers of $x$) is a finitely generated $R$ submodule of $O$, and so is integral
over $R$.  Thus, the characteristic polynomial of (\ref{charpoly}) has coefficients
in $R$, i.e. $\T(x)$ and $\N(x)$ are in $R$.  Since $x + \bar x \in R$, it follows that
$\bar x \in O$.  That is:

\begin{theorem}\label{properties}
The following are true:
\begin{enumerate}
\item { {All elements of an order are integral. }}
\item { {Orders are closed under the fundamental involution.}}
\end{enumerate}
\end{theorem}

A subring which is an $R$-module may be commutative, i.e. ones like $R[x]$. However,
a non-commutative $R$-submodule of $B$ which is a ring containing $1$ is an order.
Our main reference for properties of orders is \cite{reiner}.

Let $O$ be an order in $B$.  Then $O$ is called an \emph{Azumaya
algebra} if for every maximal ideal $\mathfrak{m}$ of its center
$R$, then $O/\mathfrak{m}O$ is a central simple algebra over the field
$R/\mathfrak{m}$.  The dimension of $O/\mathfrak{m}O$ is independent
of $\mathfrak{m}$, and is equal to $4$.

As noted in \cite{reiner}, orders which are Azumaya are maximal. Also, as
in \cite{reiner}, the Azumaya property is preserved under etale ring extensions.
The following is implicit in \cite{reiner}:

\begin{theorem}\label{az}\
Suppose $M$ is a finite field extension of $L$ and $S$ is the ring of integers of
$R$ in $M$.  Assume that $B$ is ramified only at infinite places of $L$.  Then:
\begin{enumerate}
\item Every maximal order of $B$ is Azumaya.
\item If $O$ is a maximal order in $B$, then $O\tensor_R S$ is a maximal order of $B\tensor_R M$.
\end{enumerate}
\end{theorem}

Let $O_1$ and $O_2$ be orders of $B$.  Then, arguing along the lines of problem 8 in Section 26 (Chapter 6) in \cite{reiner} : $O_1 \cong O_2$ if and only if $O_1$ and $O_2$ are conjugate, i.e. there is an element $b \in B$ with $O_1 = b O_2 b^{-1}$.  Furthermore, there are finitely many conjugacy classes of maximal orders in $B$.

We revert to our standard setup: Let $A = (-1,-1)_K$ be the quaternion algebra over a
field $K$, where $K$ is the real subfield of a cyclotomic field. The following shows the inevitability
of maximal orders.  Let $G$ be a finite subgroup of the multiplicative group of $A$, and let
$R$ be the ring of integers of $K$.  Set $R[G]$ equal to the set of $R$-linear
combinations of the elements of $G$.  The fact that $G$ is a finite group (i.e. closed under
products) says that $R[G]$ is a finitely generated $R$-module which is also a ring
containing $1$.  It follows that $R[G]$ is contained in some maximal order (see \cite{reiner}).
Suppose $g \in G$.  Since the norm form of $A$ is positive definite, $\N(g) = g \bar g = 1$,
and $\bar g = g^{-1} = 1/g$.  There is a sort of converse.  Suppose $O$ is an order of $A$, and set

\begin{equation}
\X (O) = \{g \in G\ :\ \N(g) = 1\}.
\end{equation}

Observe that $\X (O)$ is a group, and it is called the {\it norm one} subgroup of $O$.
The elements of $\X(O)$ are then elements of bounded length in a lattice, so
$\X(O)$ is finite, and is a finite subgroup of $A$.  Clearly, when $G$ is a subgroup of $A$,
$R[G]$ is contained in the $\X$ of any maximal order containing $R[G]$.  In summary, we
conclude:

\begin{theorem}\label{exhaust}
Every finite subgroup $G$ of $A$ is contained in the norm one subgroup of some
maximal order $O$.  Conversely, the norm one subgroups of maximal orders, up to
conjugacy and up to subgroups, contain all finite subgroups of $A$.
\end{theorem}

If $G$ is a finite subgroup of $A$, then $R[G]$ is one of two types.  When $G$ is abelian
(and therefore cyclic by (2) of \ref{am}), $R[G]$ is a rank 2 commutative order.  When $G$ is nonabelian, $R[G]$ is an order of $A$.  We say that $G$ is {\sl full} if $R[G]$ is a maximal
order of $A$.

There are two groups which are always full.  If $G$ is the Clifford group (of order $48$)
or the icosohedral group (of order $120$), then $G$ is known to be a maximal closed subgroup
of $SL(2,\C))$, $\C$ = complex numbers (see \cite{springer}). Since there is an isomorphism of norm one elements
from the quaternions into $SL(2,\C)$ by a map we have not discussed, we conclude:

\begin{theorem}\label{full}
Suppose $G$ is a finite subgroup of $A$, where $G$ is either the Clifford group of order $48$ or the icosohedral group of order $120$.  Then $R[G]$ is a maximal order of $A$.
\end{theorem}

\section{Examples}
We are now ready for examples.  Sample Magma code will be given when convenient.  Properties
of orders needed to generate examples will also be formulated when needed.

%\subsection{$K = \Q$} \label{sec:Q}

We first consider $K = \Q$.  The ordinary quaternion algebra $A = (-1,-1)_{\Q}$ has one conjugacy class
of maximal orders.  A representative $O$ is generated by the elements $\{i,j,k\}$
and the element $1/2(1 - i -j -k)$ of order $3$.  The norm one group $G$ of $O$ has order
$24$.  In addition, $G$ is full, but $R[G] = O$ is not Azumaya.  Theorem \ref{az} does not apply; when
$O$ is extended to the field $\Q({\sqrt 2})$, the resulting order is not maximal. All of this can be
verified by the following Magma code, which is complicated by the fact that Magma
must distinguish between the rational numbers and the number field of degree one.

\begin{verbatim}
R<t> := PolynomialRing(Rationals());
Q := NumberField(t-1 : DoLinearExtension);
A<i,j,k> := QuaternionAlgebra(Q,-1,-1);
O := MaximalOrder(A);
assert #ConjugacyClasses(O) eq 1;
G,f := NormOneGroup(O);
assert #G eq 24;
M := [f(g) : g in G];
Test := Order(M);
assert Test eq O;
C<r> := CyclotomicField(8);
s := r + 1/r;
K := NumberField(MinimalPolynomial(s));
temp<a> := QuadraticField(2);
bool,phi := IsSubfield(temp,K);
sqrt := phi(a);
B<u,v,w> := QuaternionAlgebra(K,-1,-1);
Gens := Generators(O);
SET := [];
for m in [1..#Gens] do
Append(~SET,B!ElementToSequence(Gens[m]));
end for;
TEST := Order(SET);
IsMaximal(TEST); % answer no
elt := (u - v)/sqrt; % elt is not in SET
max := Adjoin(TEST,elt);
assert IsMaximal(max);
\end{verbatim}

%\subsection{$K = \Q(\sqrt{2})$} \label{sec:qroot2}

We now take $K = \Q(\sqrt{2})$.  The algebra $A = (-1,-1)_K$ has one conjugacy class of maximal orders.
The norm one group of a maximal order is the Clifford group $G$ of order $48$, and is generated
as described in Theorem \ref{sl}.  Observe that $G$ is full, as noted in Theorem \ref{full}.  The Magma
code is similar enough to the case $K = \Q$ that we pass on details.

%\subsection{$K = \Q(\sqrt{5})$} \label{sec:qroot5}

Now, take $K = \Q(\sqrt{5})$, so the field $K$ is the real subfield of the cyclotomic field of $5$-th roots of
unity.  The algebra $A = (-1,-1)_K$ has one conjugacy class of maximal orders.
As predicted by Theorem \ref{sl}, the norm one group of a maximal order is the group $G = {\rm SL}(2,5)$ of
order $120$.  This group is full.  As $A$ is ramified only at infinite primes,
Theorem \ref{az} applies; $R[G]$ is maximal, and its extension to any number field
containing $K$ remains Azumaya and maximal, with norm one group of order $120$ (since
$G$ is maximal closed in $SL(2,\C)$).  Again, the Magma code is close enough to
the case $K = \Q$ that we omit details.  We note that nothing would change
if instead $K$ is considered the real subfield of the field of $10$-th roots of unity;
this is actually the context of Lemma \ref{change}.

\section{Cyclotomic fields}

We now change notation for convenience.  We will set $A := (-1,-1)_n$ to be
the quaternion algebra over $K$ where $K$ is the real subfield of the cyclotomic
field of $n$-th roots of unity.  The previous cases were for $n = 1,8,5,10$.

%\section{$n = 12$} \label{sec:case12}

We now consider $n = 12$.  In $A = (-1,-1)_{12}$ there are two equivalence classes of maximal orders.  Let
$O$,$V$ be representatives of the two classes.  Both have norm one group of order
$24$, but these groups are not isomorphic.  Let $G$ be the norm one group of $O$ and
$H$ the norm on group of $V$.  We arrange it so that $G$ is the generalized quaternion
group $G_{24}$; it is generated by elements $x$,$y$ so that $x$ has order $12$, $y$ has
order $4$, $yxy^{-1} = x^{-1}$, and $y^2 = x^6$ is central.  On the other hand, $H$ is the subgroup of the
Clifford group of order $24$ as described in section \ref{sec:finite}.

The groups $G$ and $H$ can be
distinguished by the fact that $G$ has an element of order $12$ and $H$ does not.  The group
$G$ is full, but its order $O$ is not Azumaya (the center $K$ does not hold a square
root of $2$); the group $H$ is not full, as $R[H]$ is not maximal.
If $S$ is the ring of integers of $24$-th roots of unity, then $O\tensor_R S$ is maximal, but its norm
one group is the Clifford group of order $48$.  Thus, extensions of maximal orders can have
larger norm one groups. The following is sample Magma code for verifying all of the above:

\begin{verbatim}
C<r> := CyclotomicField(12);
s := r + 1/r;
K := NumberField(MinimalPolynomial(s));
A<i,j,k> := QuaternionAlgebra(K,-1,-1);
O := MaximalOrder(A);
List := ConjugacyClasses(O);
assert #List eq 2;
G,f := NormOneGroup(List[1]);
H,h := NormOneGroup(List[2]);
assert Order(G) eq 24; assert Order(H) eq 24;
M := [f(g) : g in G];
TEST1 := Order(M);
IsMaximal(TEST1); %answer true
N := [h(g) : g in H];
TEST2 := Order(N);
IsMaximal(TEST2); %answer false
Gens := Generators(List[1]);
C1<r1> := CyclotomicField(24);
s1 := r1 + 1/r1;
K1 := NumberField(MinimalPolynomial(s1));
IsSubfield(K,K1);
B := QuaternionAlgebra(K1,-1,-1);
SET := [];
for g in Gens do
Append(~SET,B!ElementToSequence(g));
end for;
test := Order(SET);
IsMaximal(test); %answer yes
#NormOneGroup(test); %answer 48
\end{verbatim}

%\subsection{$n = 16$} \label{sec:case16}

For $n = 16$, there are two conjugacy classes of maximal orders in $(-1,-1)_{16}$; their norm one groups
have orders $48$ and $32$.  The first is the Clifford group, and the second the generalized
quaternion group of order $32$.  The Clifford group is full; the other group is not.  The Magma code is similar to the case $n = 16$.

%\subsection{$n = 20$} \label{sec:case20}

When $n = 20$, the algebra $K = (-1,-1)_{20}$ has three conjugacy classes of maximal orders.  The norm one groups
of these maximal orders have orders $24$, $40$, and $120$.  The group having order $120$ is $SL(2,5)$ as per Theorem \ref{sl} because the center of $K$ contains $\sqrt{5}$; this group is full and its order is Azumaya.  The group of order $40$ is the generalized quaternion group $G_{40}$; it is full and its order Azumaya.  However, this order has norm one group of order $80$ when tensored up to the field of $40$-th roots of unity.  The group of order $24$ is ${\rm SL} (2,3)$ and is not full. The Magma code is similar to the case $n = 12$.

%\subsection{$n = 24$} \label{sec:case24}

Suppose that $n = 24$.  Here something interesting happens; another group of $48$ appears.  There are three conjugacy classes of maximal orders with norm one groups of orders $16$, $48$,and $48$.  The groups of order
$48$ are the Clifford group and the generalized quaternion group $G_{48}$; these groups be distinguished since the derived subgroup of the Clifford group has index $2$ and is nonabelian while the derived subgroup of $G_{48}$ has index $4$ and is cyclic.  Another way to distinguish these groups is the fact that $G_{48}$ has an element of order $12$ while the Clifford group does not.  The maximal orders attached to these groups are Azumaya, and both groups are full.  The group of order $16$ is of course the quaternion group; it is not full.  Again Magma code is similar to the case $n = 12$.

\subsection{$n = 32, 40, 48$} \label{sec:case32}

Consider $n =32$.  Here chaos enters because Magma can barely do the computations.  There
are $58$ conjugacy classes of maximal orders.  The orders of their
norm one groups, without repetitions but sorted by decreasing order,
are: $\{64,48,32,16,8,6,4,2\}$.  The groups of order $6$,$4$,$2$ are
cyclic, and so generate commutative orders; they cannot be full.  They
are the overwhelming majority; there are $40$ conjugacy classes of
orders with norm one group of order $2$.  (An order has norm one group
of order $2$ means its only elements of norm $1$ are $\pm 1$).  The non abelian
groups of order $64$ and $32$ appear once; they are not full.

The group of order $48$ is full as predicted from earlier cases; as here $\sqrt 2$ is
in $K$.
The group of order $16$ occurs twice; the group of order $8$ three times.  They are not
full; for these groups $G$ the rings $R[G]$ are isomorphic, non-maximal, and extend
to maximal orders in distinct ways.  This tempts us to pass to a larger ring $R$, which
will be the subject of section \ref{sec:chrings}.  The group of order $48$ is the Clifford group;
the other non abelian groups of orders $64$,$32$,$16$,$8$ are generalized quaternion.

{}From this point on, except for some isolated cases which follow, we can no longer proceed
by enumerating all conjugacy classes of maximal orders.  Magma cannot do these calculations,
and we do not know what can.  We limit ourselves to the exceptional cases of $n = 40$,
$n = 48$, where the maximal orders can still be fully enumerated.

%\subsection{$n = 40$} \label{sec:case40}

When $n = 40$, there are $25$ conjugacy classes of maximal orders.  All but seven have norm one group
cyclic of orders $2,4,6,10$; these generate commutative quaternion orders.  The remaining
seven have norm one groups of orders $24, 8, 120, 48, 16, 16, 80$.  The groups of order
$120$ and $48$ are the usual $SL(2,5)$ (the center contains $\sqrt{5}$ as per Theorem \ref{sl}) and the
Clifford group (the center contains $\sqrt{2}$ as per Theorem \ref{sl}); these groups are full
and their quaternion orders are Azumaya.  The group of order $80$ is the generalized
quaternion group $G_{80}$; it is full and its quaternion order Azumaya.  The remaining
non abelian groups are not full.  The two groups of order $16$ are generalized quaternion
and isomorphic, but live in two distinct non conjugate maximal orders.

%\subsection{$n = 48$} \label{sec:case40}

The case of $n = 48$ is the last case where we can enumerate all conjugacy classes of maximal orders;
there are $39$ of them.  The majority have norm one groups which are cyclic
of orders $2,4,6$, all with multiplicities $> 1$.  Seven non abelian groups remain,
with orders $48, 16, 16, 32, 32, 96, 8, 8, 8$.  The Clifford group of order $48$ is
full; the group $G_{96}$ is full with its quaternion order Azumaya.  The remaining
groups are generalized quaternion and not full.

\subsection{$n = 56$} \label{sec:case56}

We include the case of $n = 56$ for two reasons.  For one, Magma is not able to
exhaust the list of conjugacy classes of maximal orders, so we confront
the exercise of seeing what we can construct.  Secondly, this is the first
time that there are abelian and nonabelian norm one groups of the same order;
this occurs at order $8$.

One preliminary is in order.  Let $C$ be the cyclotomic field of
$56$-th roots of unity, and
let $K$ the center of $A = (-1,-1)_{56}$.  Asking Magma for a default
maximal order of $A$ produces the Clifford group of order $48$ ($K$
contains a square root of $2$, which is named sqrt in the following
Magma code).  Since $C/K$ is quadratic, $C$ is generated over $K$ by
any element not in $K$.  Moreover, as a quadratic splitting field of
a quaternion algebra, $C$ can be embedded in $A$.  We choose to embed
$C$ in $A$ via an element $d$ so that $j d j^{-1} = d^{-1}$, $j$ the
usual quaternion unit.  For this reason, $A$ is isomorphic to any of
the crossed product algebras $A(56,j,-1)$, $A(28,j,-1)$, $A(14,j,-1)$,
corresponding to the choices of generator $d$, $d^2$, $d^4$.  Asking Magma
for a maximal order according to these descriptions results in maximal orders
with norm one groups corresponding to the generalized quaternion groups of
order $112$ and $28$.  It does not produce the generalized quaternion group
of order $56$; we do not know if this can occur.

The technique above fails for larger powers of $d$; if the resulting quaternion
order has the right group structure, it is not maximal, and enlarging to
maximal orders tends to reproduce the Clifford group.

Thus, we need a new technique, which we now outline.  Suppose $S$
is a set which generates an order which is not maximal; we take $S$
to be closed under sums and products in such a way that $O = R[S]$
is an order; $O$ is not maximal.  Then $O$ can be enlarged to an
order larger than $O$ by adjoining an element $x \in A$ so that $xs$
and $x + s$ are integral for all $s \in S$.  These conditions can
be minimized to: $\T(x \bar s) \in R$ for all $x \in S$.  Since we
may assume $1 \in S$, in particular $x$ is integral.  Since all non
abelian subgroups of $A$ contain the quaternion group of order $8$,
we may assume (after conjugation by some element of $A$) that $O$ is
the order generated by $[1,i,j,k]$.  We will attempt to tag on one
additional element in order to obtain a maximal order.

In all of the following, $c$ is a primitive element of $K$ over $\Q$.
Our elements will have coefficients which are linear combinations of powers of
$C$.  We will need the square root of $2$, which is the element sqrt
in the following Magma code.  In order to avoid elements in $O$, all of our
coefficients will be $1/2$ times polynomials in $R[c]$.

We begin by squaring elements which are expressed as even powers of
$c$; there will be rollover because $K$ over $\Q$ has finite degree,
but the resulting elements will be sums of even powers of $c$ with
integer coefficients.  We attempt to make the result into the norm
of an integral element by adjusting the coefficients so as to become
$0$ mod $4$.  For this, we need sqrt in case a coefficient is $2$ mod
$4$.  Once $x$ is integral, the shape of our elements will guarantee that
$x s$ and $x + s$ are integral for all $s$ in $\{i,j,k\}$.
The method is illustrated by the choice of $b1$ and $b2$ in the
following Magma code.  In each case, we started with the square of the
constant term.  Since our methods here are hit and miss, we will not
elaborate further.  Luck intervenes when adjoining one more element
to $O$ produces a maximal order.

\begin{verbatim}
C<r> := CyclotomicField(56);
rt := r + 1/r;
K<c> := NumberField(MinimalPolynomial(rt));
bool,phi := IsSubfield(K,C);
C1<r> := RelativeField(K,C);
A<i,j,k> := QuaternionAlgebra(K,-1,-1);
M := MaximalOrder(A);
#NormOneGroup(M) eq 48; //answer true
d := Embed(C1,A);
O1 := Order([d,j]);
#NormOneGroup(O1); // answer 112
O2 := Order([d^2,j]);
#NormOneGroup(O2); // answer 112
O3 := Order([d^4,j]);
#NormOneGroup(O3); // answer 28
temp<b> := QuadraticField(2);
bool,phi := IsSubfield(temp,K);
sqrt := phi(b); //Square root of 2
b1 := (1/2)*(c^14 + c^2 + 1) + (1/2)*sqrt*(c^5 + c^4)*i \\
+ (1/2)*c^2*j + (1/2)*sqrt*k;
O := Order([1,i,j,k]);
Test := Adjoin(O,b1);
IsMaximal(Test); //answer yes
#NormOneGroup(Test) eq 16;
b2 := (1/2)*(c^6 + c^2 + 1) + (1/2)*sqrt*(c^4 + c^3 +c )*i  \\
+ (1/2)*c^4*j + (1/2)*c^2*k;
T2 := Adjoin(O,b2);
IsMaximal(T2);
#NormOneGroup(T2) eq 8;
IsAbelian(NormOneGroup(T2));
// the last order T2 achieves the quaternion group  of order 8
\end{verbatim}

We still need to construct maximal orders with norm one group which is cyclic
of order $8$.  For this, we use another form of our quaternion algebra
{\hbox {$A = (-1,-1)_K$}.  By section \ref{sec:cyc}, $A$ is isomorphic to $(-1,-u)$ for
$u$ a norm from $K(\sqrt{-1})$ to $K$.  This requires $u$ to be a sum of two
squares in $K$.  We will restrict attention to $u$ a sum of two even
powers of $c$.  Magma code describes the inspection over a suitable search
space:

\begin{verbatim}

C<r> := CyclotomicField(56);
rt := r + 1/r;
K<c> := NumberField(MinimalPolynomial(rt));
bool,phi := IsSubfield(K,C);
C1<r> := RelativeField(K,C);
A<i,j,k> := QuaternionAlgebra(K,-1,-1);
B<u,v,w> := QuaternionAlgebra(K,-1,elt);
for n in [5..10] do
for m in [1..10] do
elt := -(c^(2*n) + c^(2*m));
B<u,v,w> := QuaternionAlgebra(K,-1,elt);
O := MaximalOrder(B);
G,f := NormOneGroup(O);
if IsAbelian(G) and #G eq 8 then
print n,m;
break;
end if;
end for;
end for;
\end{verbatim}

There are successes when $(n,m) = (8,1), (9,2), (10,3)$.  We have not checked whether
the three resulting maximal orders are inequivalent once rendered into $A$.

\subsection{$n = 64$} \label{sec:case64}

Magma is unable to list all conjugacy classes of maximal orders when $n = 64$;
there are a screaming number of them.  Consequently, we content ourselves with
assembling what maximal orders we can with the techniques on hand.  This will be
enough to construct all the finite groups that we expect.  We are not able to
decide whether norm one groups of maximal orders can be cyclic of order $16$
in this case.

We would expect all the finite subgroups of $B = (-1,-1)_{32}$ to occur
in $A = (-1,-1)_{64}$.  This can be verified by tensoring the maximal orders
of $B$ with $R$, the ring of integers of $64$-th roots of $1$.  These extensions
are maximal orders by Theorem \ref{az}, and their extensions have the same
norm one group (a Magma verification).  This accounts for all the groups
that arose at $n = 32$, but cannot account for the cyclic group of order $8$,
or the generalized quaternion group $G_{128}$.

To account for the group of order $128$, we argue as follows.  Let $C$ be
the cyclotomic field of $64$-th roots of unity.  Then $C \cong K(\sqrt{-1})$,
and $K(\sqrt{-1}) \cong K(i)$ for the quaternion unit $i$.  It follows that
there is a $64$-th root of unity $d$ in $A$ of form $a + bi$, and is therefore
conjugated to its inverse by the quaternion unit $j$.  Then the order $M$
generated by the set $[d,j]$ contains $G_{128}$ since $d$ and $j$ generate
this group.  This order $M$ is not maximal, but any maximal order containing
it will have norm one group $G_{128}$.  This can all be verified by Magma;
we do not list the corresponding code here.

Let $O$ be the order generated by $S = [1,i,j,k]$.  $O$ is not maximal, but its
norm one group does contain the quaternion group of order $8$, which we
denote by $Q_8$ in what follows.  Every non abelian group in $A$ contains an
isomorphic copy of $Q_8$, so, up to conjugacy, every order with non abelian
group of norm one elements can be assumed to contain $O$ (see later discussion
of the Skolem-Noether theorem in Section \ref{sec:sknt}).  We can then apply the
methods of Section \ref{sec:case56} to adjoin one element to $S$ whose
coefficients are half integers and attempt to hit on maximal orders with
non abelian norm one group.  The following Magma code indicates some of
these successes.

\begin{verbatim}
C<r> := CyclotomicField(64);
rt := r + 1/r;
K<c> := NumberField(MinimalPolynomial(rt));
bool,phi := IsSubfield(K,C);
C1<r> := RelativeField(K,C);
A<i,j,k> := QuaternionAlgebra(K,-1,-1);
M := MaximalOrder(A);
assert #NormOneGroup(M) eq 48;  \\ this accounts for the Clifford group
d := Embed(C1,A);
O1 := Order([d,j]);
O2 := Order([d^2,j]);
O3 := Order([d^4,j]);
M1 := MaximalOrder (O1);
M2 := MaximalOrder (O2);
M3 := MaximalOrder (O3);
assert #NormOneGroup(M1) eq 128;
assert #NormOneGroup(M2) eq 64;
assert #NormOneGroup(M3) eq 32;
b3 := 1/2*(c^15+c^14+c^11+c^9+c^6+c^4+c^3)+1/2*(c^12+c^9+c^6+c^4+c^3)*i;
O3a := Order ([d^4,j,b3]);
M3a := MaximalOrder (O3a);
assert #NormOneGroup(M3a) eq 32;
ba := 1/2*(c^15+c^14+c^13+c^11+c^8+c^7+c^5)+1/2*(c^14+c^11+c^8+c^7+c^5)*i;
O4a := Order ([d^8,j,ba]);
M4a := MaximalOrder (O4a);
assert #NormOneGroup (M4a) eq 16;
bb := 1/2*(c^14+c^12+c^8+c^6)+1/2*(c^12+c^8+c^6)*i;
O4b := Order ([d^8,j,bb]);
M4b := MaximalOrder (O4b);
assert #NormOneGroup (M4b) eq 16;
bc := 1/2*(c^14+c^10+c^9+c^7)+1/2*(c^15+c^14+c^10+c^9+c^7)*i;
O4c := Order ([d^8,j,bc]);
M4c := MaximalOrder (O4c);
assert #NormOneGroup(M4c) eq 16;
sqrt2 := c^8 - 8*c^6 + 20*c^4 - 16*c^2 + 2;
bd := 1/2*(c^10+1)+1/2*sqrt2*c^5*i+1/2*sqrt2*(c+1)^2*j+1/2*k;
M4d := Adjoin(Order([i,j]),bd);
IsMaximal(M4d);
assert #NormOneGroup (M4d) eq 16;
ca := 1/2*(c^8+c^4+1)+1/2*((sqrt2)*(c^6+c^2))*i+1/2*c^4*j+1/2*k;
M5a := Adjoin(Order([i,j]),ca);
IsMaximal(M5a);
assert #NormOneGroup(M5a) eq 8;

test := 3*c^8 + c^6 + 3;
test := (1/2)*test +(1/2)*sqrt2 *(c^7 + c^6 + c^4 + c^3)* i + (1/2)*c^6*j+ 1/2*k;
M5b := Adjoin(Order([i,j]),test);
IsMaximal(M5b);
assert #NormOneGroup(M5b) eq 8;

test := 3*c^8 + c^2 + 3;
test := (1/2)*test + (1/2)*sqrt2 * (c^5 + c^4 + c^2 + c)*i + (1/2)*c^2*j + 1/2*k;
M5c := Adjoin(Order([i,j]),test);
IsMaximal(M5c);
assert #NormOneGroup(M5c) eq 8;

temp := 3*c^8 + sqrt2*c^2 + 3;
temp := (1/2)*temp + (1/2)*sqrt2*(c^5 + c^4 + c^2)*i + (1/2)*k;
Test := Adjoin(Order([i,j]),temp);
IsMaximal(Test);
assert #NormOneGroup(Test) eq 16;

test := 3*c^8 + c^2 + (3 + sqrt2);
test := (1/2)*test + (1/2)*sqrt2*(c^2 + c + 1)*i + (1/2)* c^2 * j + 1/2*k;
M5d := Adjoin(Order([i,j]),test);
IsMaximal(M5d);
assert #NormOneGroup(M5d) eq 8;

test := c^8 + (1 + sqrt2)*c^2 + 3;
test := (1/2)*test + (1/2)*(sqrt2)*(c^6 + c^4 + c)*i + (1/2)*c^2*j + 1/2*k;
M5e := Adjoin(Order([i,j]),test);
IsMaximal(M5e);
assert #NormOneGroup(M5e) eq 8;
\end{verbatim}

It can be verified that all of the quaternion orders constructed above are
distinct up to conjugacy.  This produces orders which are not extended
from a lower value of $n < 64$, and indicates why it would be hopeless
to try to evaluate all conjugacy classes of maximal orders at $n = 64$.

We have still not accounted for the cyclic group of order $8$.  We can do this
by the method that worked at $n = 56$, namely realizing $A$ as isomorphic
to the algebra $B = (-1,-elt)$ where $elt = c^{2n} + c^{2m}$; and $c$ the primitive
element of $K$ given in the Magma code above.  We tried this for
many values of $(n,m)$ between $0$ and $100$, and produced $15$ inequivalent
maximal orders with norm one group cyclic of order $8$.  We have no idea how
many of these there can be.  We also do not know if the cyclic group of
order $16$ can occur as the norm one group of a maximal order.  One would conjecture
that it should occur at $n = 128$, but at that level Magma is essentially unable
to produce a single maximal order.

\section{The Skolem-Noether Theorem} \label{sec:sknt}

Many of the consequences of the Skolem-Noether theorem have been used up to this
point; we include this discussion for the sake of completeness.  The following is
a standard tool in the theory of central simple algebras.

\begin{theorem}\label{skolem-noether} (Skolem-Noether)
Let $K \subset B \subset A$ where $B$ is a simple subalgebra of the central
simple $K$-algebra $A$.  Then every $K$-isomorphism $\phi$ of $B$ onto a subalgebra
$\tilde{B}$ of $A$ extends to an inner automorphism $A$; that is, there is an
invertible element $a \in A$ such that
\begin{equation}
\phi(b) = a b a^{-1},\ \   b \in B
\end{equation}
\end{theorem}

For proof see Theorem 7.21 of \cite{reiner}.  In particular, we have the following corollaries.

\begin{cor} Let $A$ be a central simple $K$-algebra and suppose
$L/K$ is a commutative Galois subfield of $A$.  Let $G = \Gal(L/K)$
be the Galois group of $L$ over $K$.  Then every $\sigma \in G$ is
realized by an inner automorphism of $A$.
\end{cor}

\begin{proof}
This is Theorem \ref{skolem-noether} applied to the simple subalgebra $B = L$.
\end{proof}

\begin{cor} Let $A = (-1,-1)_K$ as before, and suppose $G$, $H$ are finite subgroups
of $A$ which are isomorphic.  Let $\phi$ be a group isomorphism of $G$ onto $H$.
Then $\phi$ is realized by an inner automorphism of $A$, which also
induces a ring isomorphism of $R[G]$ onto $R[H]$.
\end{cor}

\begin{proof}
There are two cases.  Let $K[G]$ be the set of $K$-linear combinations of
elements of $G$.  If $G$ is abelian, then $K[G] \cong K[H]$ is an isomorphism of
commutative subfields induced by $\phi$.  Otherwise, $K[G] \cong K[H] \cong A$.
In any case, $\phi$ is realized by a conjugation which moves $G$ onto $H$.
\end{proof}

\begin{cor} Suppose $O_1$ and $O_2$ are orders of $A$ that are isomorphic as
$R$-modules.  Then $O_1$ and $O_2$ are conjugate in $A$.
\end{cor}

\begin{proof}
An isomorphism $f :O_1 \rightarrow O_2$ extends to an isomorphism
of $A = O_1 \tensor_R K \rightarrow O_2 \tensor_R K = A$.  This isomorphism
is achieved by a conjugation in $A$ which thus maps $O_1$ onto $O_2$.
\end{proof}

\section{Changing Rings} \label{sec:chrings}

Suppose that $A = (-1,-1)_K$ and that $R$ is the ring of integers of $K$.  We are interested in
the structure of extended modules over $S = R[1/2]$, the localization of
$R$ which makes powers of $2$ into units.  (Warning: this is not the same
as the localization of $R$ at prime ideals over $(2)$ ).  If nothing else, this
drastically reduces the number of inequivalent maximal orders.

Let $M$ be a maximal order of $A$ whose norm one group contains the nonabelian
group $G$.  Then $G$ contains an isomorphic copy of $Q_8$, and thus by
the Skolem-Noether theorem we may assume (after a conjugation) that $M$ contains
the basis $[1,i,j,k]$.

Suppose $x \in M$ and $x = a[0]\cdot 1 + a[1]\cdot i + a[2]\cdot j +
a[3]\cdot k$ is the expression of $x$ in this basis, with the coefficients in $K$.
As $x$ is integral over $R$, we have $\T (x) = 2a[0] \in R$.
As $ix$, $jx$, $kx$ are also integral and lie in
$M$, we conclude that $2a[m] \in R$ for $m \in \{1,2,3\}$.  In any case, $2x$ is in
$O$, the order generated by $[1,i,j,k]$.  This says that the quotient group
$M \mod O$ is a $2$-group. Since $2$ is a unit is $S$, we see that $M$ and $O$ become equal
when tensored to $S$.  We have proved:

\begin{theorem}\label{ext}
Let $M$ be a maximal order of $A$ whose norm one group is nonabelian.  Up to
conjugacy, we may assume that $M$ contains $O = [1,i,j,k]$.  Then
\begin{equation}
   M \tensor_R S \cong O \tensor_R S \cong S[Q_8].
\end{equation}
In particular, all such maximal orders $M$ become conjugate over $S$.
\end{theorem}

We do not know if the maximal orders with abelian norm one groups become isomorphic over $S$, or even
if there is more than one conjugacy class of maximal orders over $S$.

\bigskip
Keywords: Quaternion Algebra, division ring, simple algebra, Galois group, automorphism, order, maximal order

%
%{\tt }

\end{document}